\documentclass{amsart}\def\distribute{}
\date{November 12, 2012}
\usepackage[usenames]{color}
\newcommand{\red}[1]{\textcolor{Black}{#1}}
\usepackage[dvips]{graphicx} 
\usepackage{verbatim,enumerate}
\ifx\distribute\undefined
\title{
    Intrinsic Invariants of Cross Caps
}
\author{%
  Masaru Hasegawa \and 
  Atsufumi Honda \and 
  Kosuke Naokawa \and 
  Masaaki Umehara \and 
  Kotaro Yamada
\thanks{%
  The second and third authors were partly supported by 
  the Grant-in-Aid for JSPS Fellows.
  The fourth and fifth authors  were 
  partially supported by Grant-in-Aid for 
  Scientific Research (A) No.~22244006, 
  and Scientific Research (B) No.~21340016,
  respectively, from the Japan Society for the Promotion of Science.}
}
\authorrunning{M.~Hasegawa et.\ al.}

\institute{%
  Masaru Hasegawa \at 
  Department of Mathematics,
  Faculty of Science,
  Saitama University,  Sakura-ku, Saitama 338-857, Japan\\
  \email{mhasegawa@mail.saitama-u.ac.jp}
\and
  Masaaki Umehara \at 
  Department of Mathematical and Computing Sciences,
  Tokyo Institute of Technology,
  2-12-1-W8-34, O-okayama Meguro-ku,
  Tokyo 152-8552, Japan\\
  \email{umehara@is.titech.ac.jp}
\and 
  Atsufumi Honda \and Kosuke Naokawa \and Kotaro Yamada
  \at 
  Department of Mathematics,
  Tokyo Institute of Technology,
  O-okayama, Meguro-ku, Tokyo 152-8551, Japan \\
  \email{10d00059@math.titech.ac.jp,
         naokawa1@is.titech.ac.jp,
         kotaro@math.titech.ac.jp}
}
\usepackage{amsmath}
\usepackage{amssymb}
\else
\title{%
    Intrinsic Invariants of Cross Caps
}
\author{M.~Hasegawa}
\author{A.~Honda}
\author{K.~Naokawa}
\author{M.~Umehara}
\author{K.~Yamada}
\address[Masaru Hasegawa]{%
  Department of Mathematics,
  Faculty of Science,
  Saitama University,  Sakura-ku, Saitama 338-857, Japan}
\email{mhasegawa@mail.saitama-u.ac.jp}

\address[Masaaki Umehara]{
  Department of Mathematical and Computing Sciences,
  Tokyo Institute of Technology,
  2-12-1-W8-34, O-okayama Meguro-ku,
  Tokyo 152-8552 Japan}
\email{umehara@is.titech.ac.jp}
\address[Atsufumi Honda, Kosuke Naokawa and Kotaro Yamada]{%
  Department of Mathematics,
  Tokyo Institute of Technology,
  O-okayama, Meguro, Tokyo 152-8551, Japan%
}
\email{10d00059@math.titech.ac.jp}
\email{naokawa1@is.titech.ac.jp}
\email{kotaro@math.titech.ac.jp}

\subjclass[2010]{%
 Primary 57R45;   
 Secondary 53A05. 
}
\keywords{cross cap, curvature, isometric deformation}
\thanks{%
  The second and third authors were partly supported by 
  the Grant-in-Aid for JSPS Fellows.
  The fourth and fifth authors  were 
  partially supported by Grant-in-Aid for 
  Scientific Research (A) No.~22244006, 
  and Scientific Research (B) No.~21340016,
  respectively, from the Japan Society for the Promotion of Science.}
\fi
\usepackage{amsthm}
\theoremstyle{plain}
 \newtheorem{theorem}{Theorem}
 \newtheorem{proposition}[theorem]{Proposition}
 
 \newtheorem{corollary}[theorem]{Corollary}
\theoremstyle{definition}
 \newtheorem{definition}[theorem]{Definition}
\theoremstyle{remark}
 \newtheorem{remark}[theorem]{Remark}
 \newtheorem{example}{Example}

\newcommand{\vect}[1]{\boldsymbol{#1}}
\newcommand{\R}{\boldsymbol{R}}

\renewcommand{\phi}{\varphi}

\newcommand{\pmt}[1]{{\begin{pmatrix} #1  \end{pmatrix}}}
\newcommand{\can}{\operatorname{std}}
\newcommand{\ext}{\operatorname{ext}}
\begin{document}
\maketitle
\begin{abstract}
 It is classically known that generic smooth maps of $\R^2$
 into $\R^3$ admit only cross cap singularities. 
 This suggests that the class of cross caps might be an important 
 object in differential geometry.
 We show that the {\it standard cross cap\/} $f_{\can}(u,v)=(u,uv,v^2)$
 has non-trivial isometric deformations with infinite dimensional
 freedom.
 Since there are several geometric invariants for cross caps, 
 the existence of isometric deformations suggests that one can
 ask which invariants of cross caps are intrinsic.
 In this paper, we show  that there are three fundamental intrinsic 
 invariants for cross caps.
 The existence of extrinsic invariants is also shown.
\ifx\distribute\undefined
\keywords{%
 cross cap \and 
 curvature \and 
 isometric deformation 
}
\subclass{
  57R45  
\and
  53A05. 
}
\fi
\end{abstract}
\section{Introduction}
Let $U$ be a domain in $\R^2$ and $f:U\to \R^3$ a $C^\infty$-map.
A point $p$ $(\in U)$ is called a {\it singular point\/}
if the rank of the Jacobi matrix of $f$ at $p$ is less than $2$.
Consider such a map given by
\begin{equation}\label{eq:01}
   f_{\can}(u,v)=(u,uv,v^2),
\end{equation}
which has an isolated singular point at the origin $(0,0)$
and is called the {\it standard cross cap\/}
(see Figure~\ref{fig:deform}, left).
A singular point $p$ of a map $f:U\to \R^3$
is called a {\it cross cap\/} or a {\it Whitney umbrella\/}
if there exist local diffeomorphism $\phi$ on $\R^2$
and a local diffeomorphism $\Phi$ on $\R^3$ such that 
$\Phi\circ f=f_{\can}\circ \phi$.
Whitney proved that a $C^\infty$-map $f:U\to \R^3$ has 
a cross cap singularity at $p\in U$ if there exists 
a local coordinate system $(u,v)$ 
centered at $p$ such that 
\[
    f_{v}(0,0):=\frac{\partial f}{\partial v}(0,0)=0
\]
and three vectors
\[
    f_{u}(0,0):=\frac{\partial f}{\partial u}(0,0),\quad
    f_{uv}(0,0):=\frac{\partial^2 f}{\partial u\partial v}(0,0),\quad
    f_{vv}(0,0):=\frac{\partial^2 f}{\partial v^2}(0,0)
\]
are linearly independent.
By a rotation, a translation in $\R^3$ and a suitable orientation
preserving coordinate change of the domain $U\subset\R^2$,
we have the following Maclaurin expansion of $f$ at a cross cap 
singularity $(0,0)$ (cf.\ \cite{W} or \cite{FH})
\begin{equation}\label{eq:cross}
    f(u,v)=
        \left(
	   u,
	   uv+\sum_{i=3}^n \frac{b_i}{i!}v^i,
	   \sum_{r=2}^n \sum_{j=0}^r \frac{a_{j\,r-j}}{j!(r-j)!}u^j
	   v^{r-j}
	\right) +O(u,v)^{n+1},
\end{equation}
where $a_{02}$ never vanishes.
By orientation preserving coordinate changes $(u,v)\mapsto (-u,-v)$
and $(x,y,z)\mapsto (-x,y,-z)$, we may assume that
\begin{equation}\label{eq:a02}
    a_{02}>0,
\end{equation}
where $(x,y,z)$ is the usual Cartesian coordinate system of $\R^3$.
After this normalization \eqref{eq:a02}, one can easily verify that
all of the coefficients $a_{jk}$ and $b_i$ are
uniquely determined.
An oriented local coordinate system $(u,v)$ giving such a normal form
is called the {\it canonical coordinate system\/}
of $f$ at the cross cap singularity.
This unique expansion of a cross cap implies that
the coefficients $a_{jk}$ and $b_i$  can be considered as 
geometric invariants of the cross cap $f$.
A cross cap is called {\it non-degenerate\/} (resp.\ {\it degenerate})
if $a_{20}$ does not vanish (resp.\ does vanish).
On the other hand, a \red{real analytic} cross cap
is called {\it quadratic\/} if  $a_{jk}=0$ for 
$j+k\ge 3$ and $b_i=0$ for $i\ge 3$.
The standard cross cap is a typical example of a degenerate quadratic 
cross cap.

\begin{figure}
 \begin{center}
        \includegraphics[height=3.0cm]{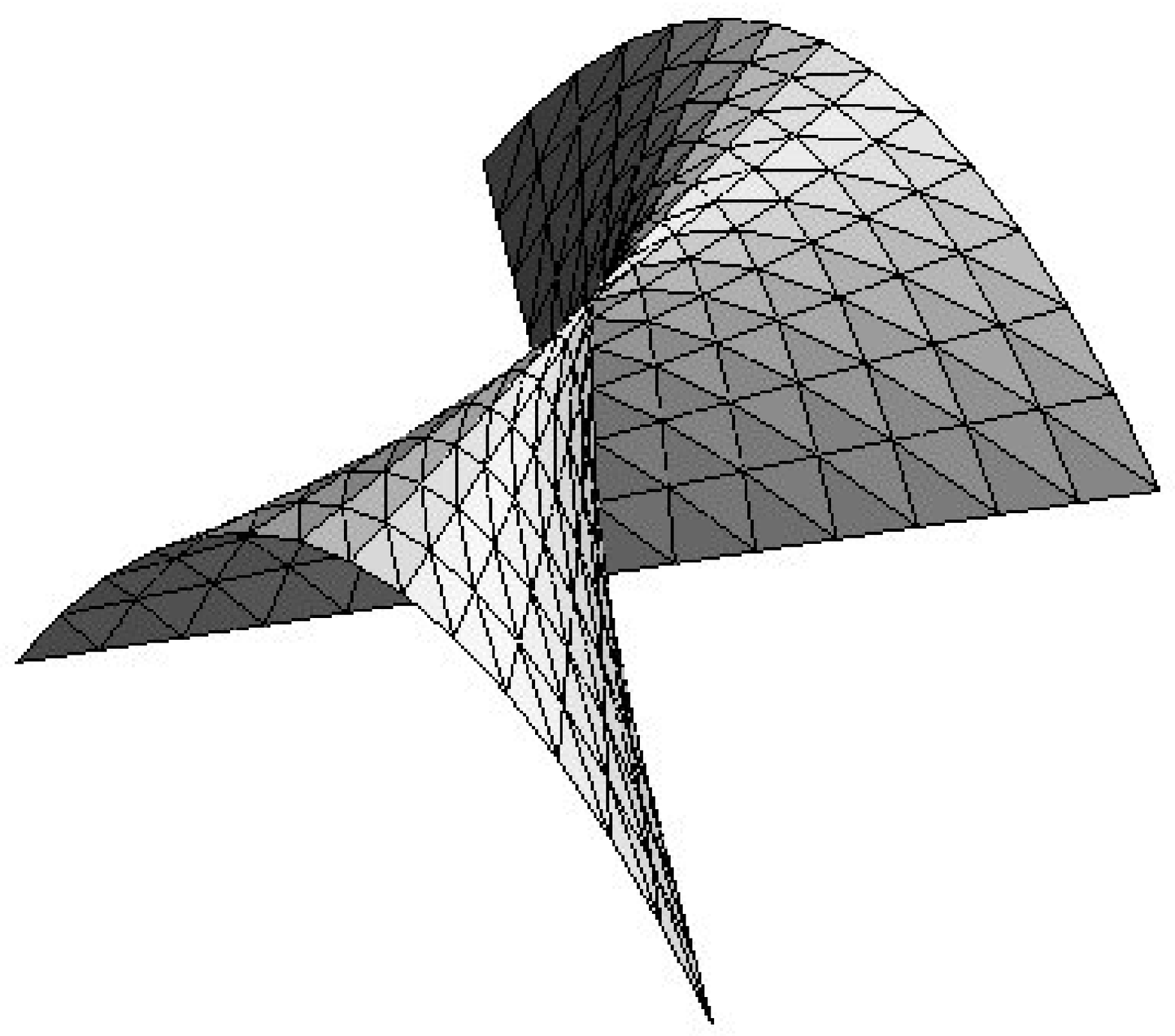}\qquad  
        \includegraphics[height=3.0cm]{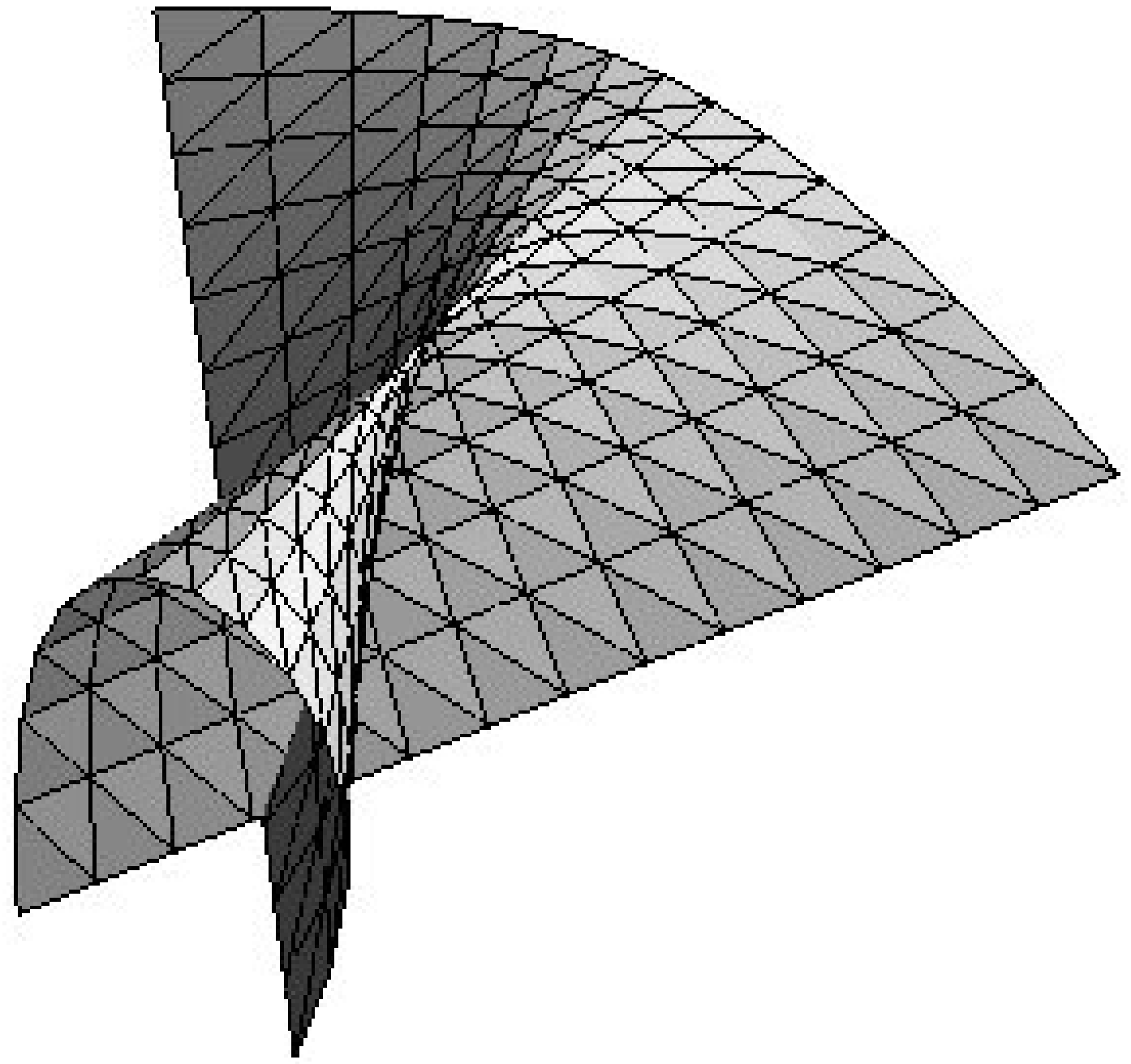}\qquad  
        \includegraphics[height=3.0cm]{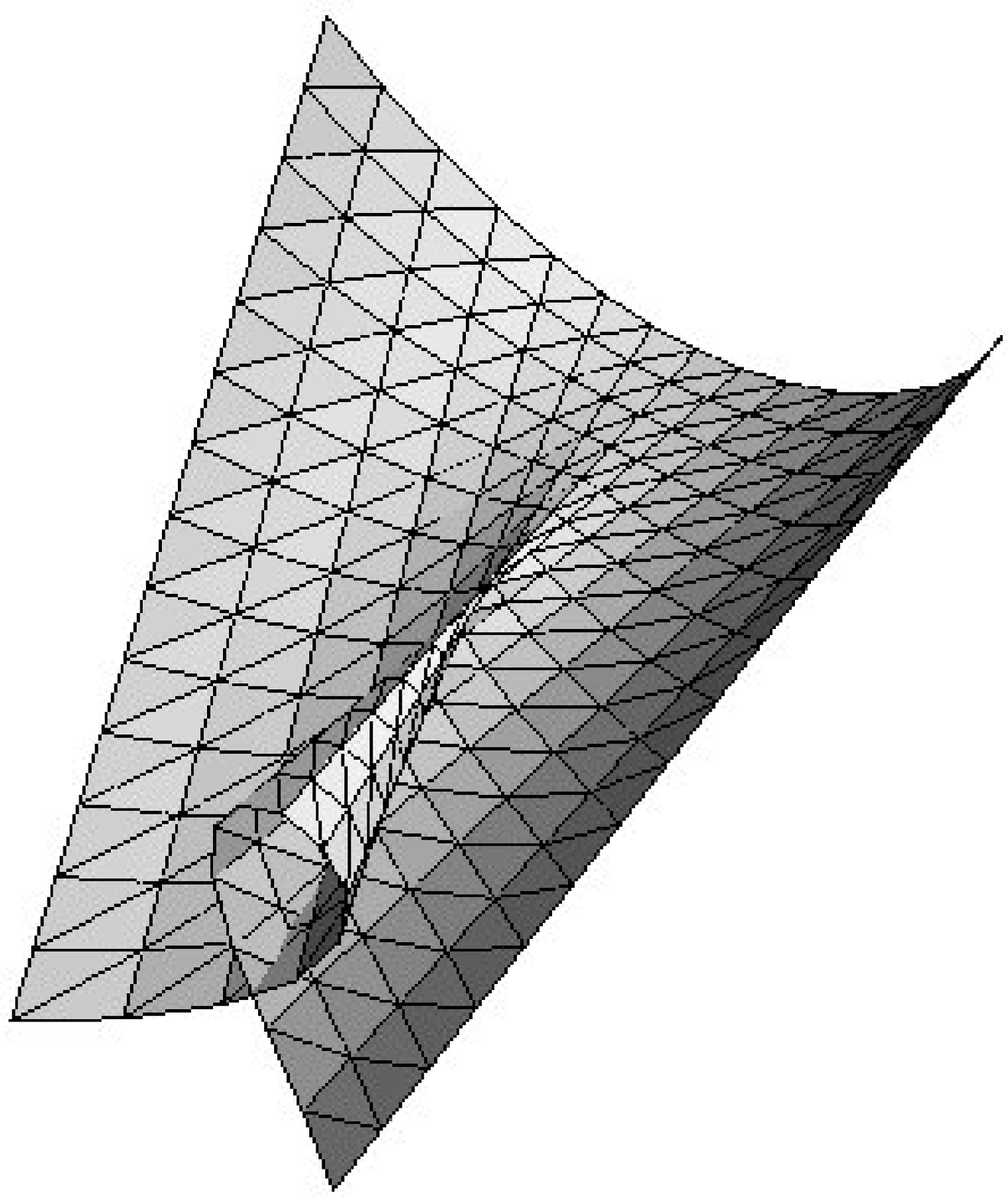}
  \caption{An isometric deformation of the standard cross cap}
\label{fig:deform}
\end{center}
\end{figure}

Let $f$ be a degenerate quadratic cross cap. 
\red{Using the classically known isometric deformations
of ruled surfaces,
we show that each 
degenerate quadratic cross cap induces a non-trivial family of 
isometric deformations with infinite dimensional freedom.} 
Moreover, using such a deformation, we show  that the invariants
$a_{03}$, $a_{12}$ and $b_3$ in \eqref{eq:cross} are extrinsic, 
namely, these invariants change according to the isometric deformation.
\red{It should be remarked that, using the same method, 
the existence of non-trivial isometric deformations are shown for other
typical singularities on surfaces,
i.e.\red{\,} cuspidal edges, swallowtails and cuspidal cross caps
(cf.\ Remark \ref{rmk:ruled}).}

The differential geometry of cross caps in $\R^3$ 
has been discussed 
by several authors 
(cf.\ \cite{BW},  \cite{FB}, \cite{FH}, \cite{FH2},
      \cite{GGS}, \cite{GS}, \cite{NT}, \cite{T} and \cite{W}).
However, the distinction between intrinsic and extrinsic
invariants has not been clearly discussed before. 
When $f:U\to \R^3$ is an immersion, 
it induces a Riemannian metric on $U$
(called the first fundamental form)
and we know that  `intrinsic' means that a given invariant 
is described in terms of this Riemannian structure on $U$.
Similarly, each cross cap induces a positive semidefinite 
symmetric tensor $ds^2$ as a pull-back of the ambient metric.
Then a given invariant of a cross cap is called {\it intrinsic\/}
if it can be described in terms of this 
positive semidefinite metric $ds^2$. 
In the case of cuspidal edges in $\R^3$, 
such an intrinsic invariant is defined as a `singular curvature'
along these singular points (cf.\ \cite{SUY1}).

We then show that $a_{02}$, $a_{20}$ and $a_{11}$ are
intrinsic invariants.
In fact, if $a_{20}$ is negative, then the Gaussian curvature 
of a given cross cap is negative and having no lower bound.
On the other hand, if $a_{20}$ is positive,
then the Gaussian curvature is not bounded from below nor from above.
Fukui and the first author \cite{FH} found an important concept of
the `focal conic' of a cross cap,
as a section of its caustic by the normal plane (see 
the explanation after \eqref{eq:parallel} and 
also \cite{FH2}).
They also showed that focal conics have the expression
\[
    y^2+2a_{11}yz - (a_{20}a_{02}-a_{11}^2)z^2 +a_{02}z = 0.
\]
The focal conic is a hyperbola (resp.\ an ellipse)
if and only if $a_{20}$ is positive (resp.\ negative).
Since we have seen that $a_{02}$, $a_{20}$ and $a_{11}$ are intrinsic,
we can say that focal conics live in the intrinsic geometry of cross
caps, although caustics themselves are extrinsic objects. 
It should be remarked that the Gauss-Bonnet type formula for 
closed surfaces which admit only cross cap singularities 
in $\R^3$ has no defection at each singular point
(cf.\ Kuiper \cite[p.\ 92]{K}).

\section{Isometric deformations of degenerate quadratic cross caps}
\label{sec:deform}

As mentioned in the introduction, a quadratic cross cap
can be expressed as
\[
    \left(u,uv,\frac12(a_{20}u^2+2a_{11}uv+a_{02}v^2)\right)
    \qquad (a_{02}>0),
\]
where $(u,v)$ is the canonical coordinate system.
\red{It should be remarked that
the set of self-intersections 
lies in a straight line (see Theorem \ref{thm:normal}).}

As defined in the introduction,
the cross cap $(0,0)$ is degenerate if
$a_{20}$ vanishes.
A degenerate quadratic cross cap
has the following expression
\begin{equation}\label{eq:f0}
    f_0(u,v):=
    \frac12\left(0,0,a_{02}v^2\right)+
    u\left(1,v,a_{11}v\right)
    \qquad (a_{02}>0).
\end{equation}
In particular, it is a ruled surface.
The first fundamental form 
\[
    ds^2:=E_0\,du^2+2F_0\,du\,dv+G_0\, dv^2
\]
of $f_0$ is given by
\begin{align*}
   E_0&:=(f_0)_u\cdot (f_0)_u=1+(1+a_{11}^2)v^2, \\
   F_0&:=(f_0)_u\cdot (f_0)_v=(1+a_{11}^2)uv+a_{02}a_{11}v^2, \\
   G_0&:=(f_0)_v\cdot (f_0)_v=(1+a_{11}^2)u^2+2 a_{02}a_{11}uv+a_{02}^2v^2,
\end{align*}
where the dot indicates the canonical inner product of $\R^3$.

\begin{definition}
 Let $U$ be a domain in $(\R^2;u,v)$ containing the origin,
 and let $f_i:U\to \R^3$ ($i=0,1$) be two $C^\infty$-maps having
 a cross cap singularity at $(0,0)$.
 If $f_0$ and $f_1$ satisfy
 \begin{gather*}
    (f_0)_u\cdot (f_0)_u=(f_1)_u\cdot (f_1)_u,\quad 
    (f_0)_u\cdot (f_0)_v=(f_1)_u\cdot (f_1)_v,\\
    (f_0)_v\cdot (f_0)_v=(f_1)_v\cdot (f_1)_v,
 \end{gather*}
 then we say that $f_0$ is {\it isometric\/} to $f_1$.
 On the other hand, let $f_t:U\to \R^3$ ($|t|<\epsilon$),
 be a smooth $1$-parameter family of $C^\infty$-maps having
 a cross cap singularity at $(0,0)$, where $\epsilon$ is 
 a positive constant.
 Then 
 $\{f_t\}_{|t|<\epsilon}$ is called an {\it isometric deformation\/}
 if each $f_t$ is isometric to $f_0$.
 An isometric deformation of $f_t$ is {\it non-trivial\/}
 if each $f_t$ is not congruent to $f_0$.
\end{definition} 

\red{It is classically known that ruled surfaces admit non-trivial
isometric deformations in general. 
As pointed out in Remark \ref{rmk:ruled},
several singularities (i.e.\ cross caps, cuspidal edges, swallowtails
and cuspidal cross caps) may admit isometric deformations as
ruled surfaces.
The following assertion gives a characterization
of the degenerate quadratic cross caps:}

\begin{theorem}\label{thm:deform}
 Let $c(s)$ $(|s|<\pi/2)$ be a regular curve in the unit sphere 
 $S^2(\subset \R^3)$ with arc length parameter. 
 We set
 \begin{equation}\label{eq:xi}
     \xi(v):=\sqrt{1+(1+a_{11}^2)v^2}\, \hat c(v),\qquad 
     \hat c(v):=c\left(\arctan\bigl(v\sqrt{1+a_{11}^2}\bigr)\right),
 \end{equation}
 for each $v\in \R$, and
 \begin{equation*}
    \gamma(v):=\frac{a_{02}}{{1+a_{11}^2}}
     \int_0^v t B(t)dt,\qquad
     B(v):=a_{11}\xi'(v)+\xi(v)\times \xi'(v),
 \end{equation*}
 where the prime means the derivative with respect to $v$
 and $\times$ denotes the vector product in $\R^3$.
 Then a ruled surface $f_c:\R^2\to \R^3$ defined by
 \[
    f_c(u,v):=\gamma(v)+u \xi(v)
 \]
 has a cross cap singularity 
 at the origin such that $f_c$ is isometric to a degenerate
 quadratic cross cap $f_0$.
 Moreover, let $c_i(s)$ $(|s|<\pi/2\,;~i=1,2)$ be 
 two regular curves in $S^2$ with arc length parameter. 
 Then $f_{c_1}$ is congruent to $f_{c_2}$ 
 if and only if $c_1$ is congruent to $c_2$ in $S^2$.
 \red{In this correspondence 
$c \mapsto f_c$ between spherical curves and cross caps,
 the initial degenerate quadratic cross cap corresponds to 
the geodesic in $S^2$.} 
More precisely,
 $f_c$ is congruent to $f_0$ as in \eqref{eq:f0}
 if and only if  $c(s)$ is a geodesic in $S^2$.
\end{theorem}

\begin{proof}
 By \eqref{eq:xi}, we have that 
 \begin{align}
           \xi(v)\cdot \xi(v)&=1+(1+a^2_{11})v^2, \\ 
     {\xi'}(v)\cdot {\xi'}(v)&=1+a^2_{11}, \label{eq:xi1}\\
      {\xi}(v)\cdot {\xi'}(v)&=(1+a^2_{11})v, \label{eq:xi2}
 \end{align}
 where we used the fact that $\hat c'(v)$ is orthogonal to 
 $\hat c(v)$. 
 Since $\xi(v)\times \xi'(v)$ is orthogonal to $\xi(v)$ and $\xi'(v)$, 
 the equations \eqref{eq:xi1} and \eqref{eq:xi2} yield that
 \begin{equation*}
  \xi(v)\cdot B(v)=a_{11}(1+a^2_{11})v,\qquad 
   \xi'(v)\cdot B(v)=a_{11}(1+a^2_{11}).
 \end{equation*}
 Finally, we have
 \[
    B\cdot B=|\xi\times \xi'|^2+a^2_{11} |\xi'|^2
    =|\xi|^2|\xi'|^2-(\xi\cdot \xi')^2+a^2_{11} |\xi'|^2
    =(1+a^2_{11})^2.
 \]
From now on,
we denote $f:=f_c$
for the sake of simplicity. 
 One can  prove that $f_u\cdot f_u$,\,\, $f_u\cdot f_v$ 
 and $f_v\cdot f_v$ coincide with 
 $E_0$, $F_0$ and $G_0$ respectively, using the above relations. 

 The unit normal vector field $\nu(u,v)$ is given by
 \[
    \nu(u,v)=
        \frac{1}{\delta}
          \biggl(-
           \left(a_{02}v 
                 \sqrt{1+(1+a^2_{11})v^2}
           \right) \vect{e}(v)
           +
         \bigl((1+a^2_{11})u+a_{02}a_{11}v\bigr) \vect{n}(v)\biggr),
 \]
 where 
 \[
    \delta:=\sqrt{1+a^2_{11}}
    \sqrt{(1+a^2_{11})u^2+2a_{02}a_{11}uv+a^2_{02}v^2(1+v^2)},
 \]
 $\vect{e}=dc/ds$ and $\vect{n}=c\times \vect{e}$.  
 By a straightforward calculation, 
 the second fundamental form  $L\,du^2+2M\,du\,dv+N\,dv^2$ 
 of $f$ is given by 
 \begin{equation}\label{eq:LM}
 \begin{aligned}
    &L:=0,\quad  
     M:=-\frac{a_{02}\sqrt{1+a^2_{11}}v}{\delta}, \\
    &N:=\frac{a_{02}\sqrt{1+a^2_{11}}}{\delta}u+
\frac{\delta \kappa(v)}{(1+(1+a^2_{11})v^2)^{3/2}},
 \end{aligned}
 \end{equation}
 where $\kappa(s)$ is the geodesic curvature of $c(s)$.
 Let $c_i(s)$ $(|s|<\pi/2;\,\, i=1,2)$ be two regular curves in 
 $S^2$ with arc length parameter, 
 and $\kappa_i(s)$ the geodesic curvature function of $c_i(s)$. 
 Then $f_{c_1}$ is congruent to $f_{c_2}$ 
 if $c_1$ is congruent to $c_2$ in $S^2$,
 since \eqref{eq:LM} implies that
 the second fundamental form of $f_{c_1}$
 coincides with that of $f_{c_2}$ if and only if $\kappa_1$ coincides 
 with $\kappa_2$.
 Finally, as seen in the following corollary, 
 degenerate quadratic cross caps correspond to the
 great circles, so we get the assertion.
\end{proof}

\begin{example}
Take a constant $\kappa$ and set
\[
    c_\kappa(s):=\frac{1}{\mu^2}
    \biggl(
       \kappa^2+\cos(\mu s),
       \mu \sin(\mu s),
       \kappa\bigl(1-\cos(\mu s)\bigr)
    \biggr)\quad
    \left(|s|<\frac{\pi}2,~ \mu:=\sqrt{1+\kappa^2}\right),
\]
which gives a circle in $S^2$ with arc length parameter and
of constant geodesic curvature $\kappa$.
Then it produces a deformation of the standard cross cap, 
where $c_0$ corresponds to $f_{\can}$ as in \eqref{eq:01}.
Figure~\ref{fig:deform} indicates the cross caps corresponding to 
$\kappa=0$, $1$ and $3$, respectively.
\end{example}

\begin{figure}[htb]
 \begin{center}
  \begin{tabular}{ccc}
        \includegraphics[height=3cm]{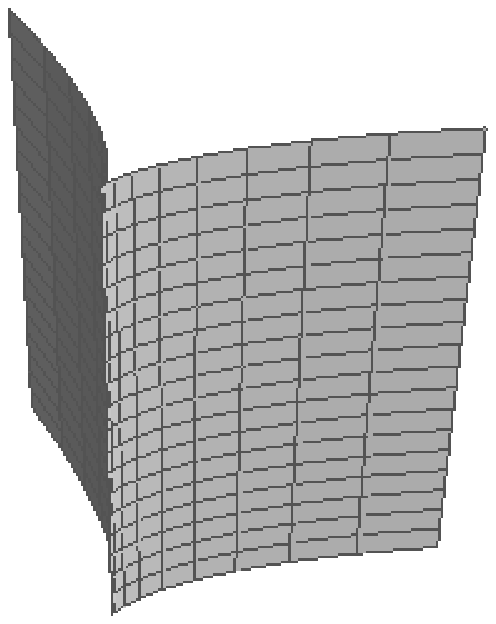}&
        \includegraphics[height=3cm]{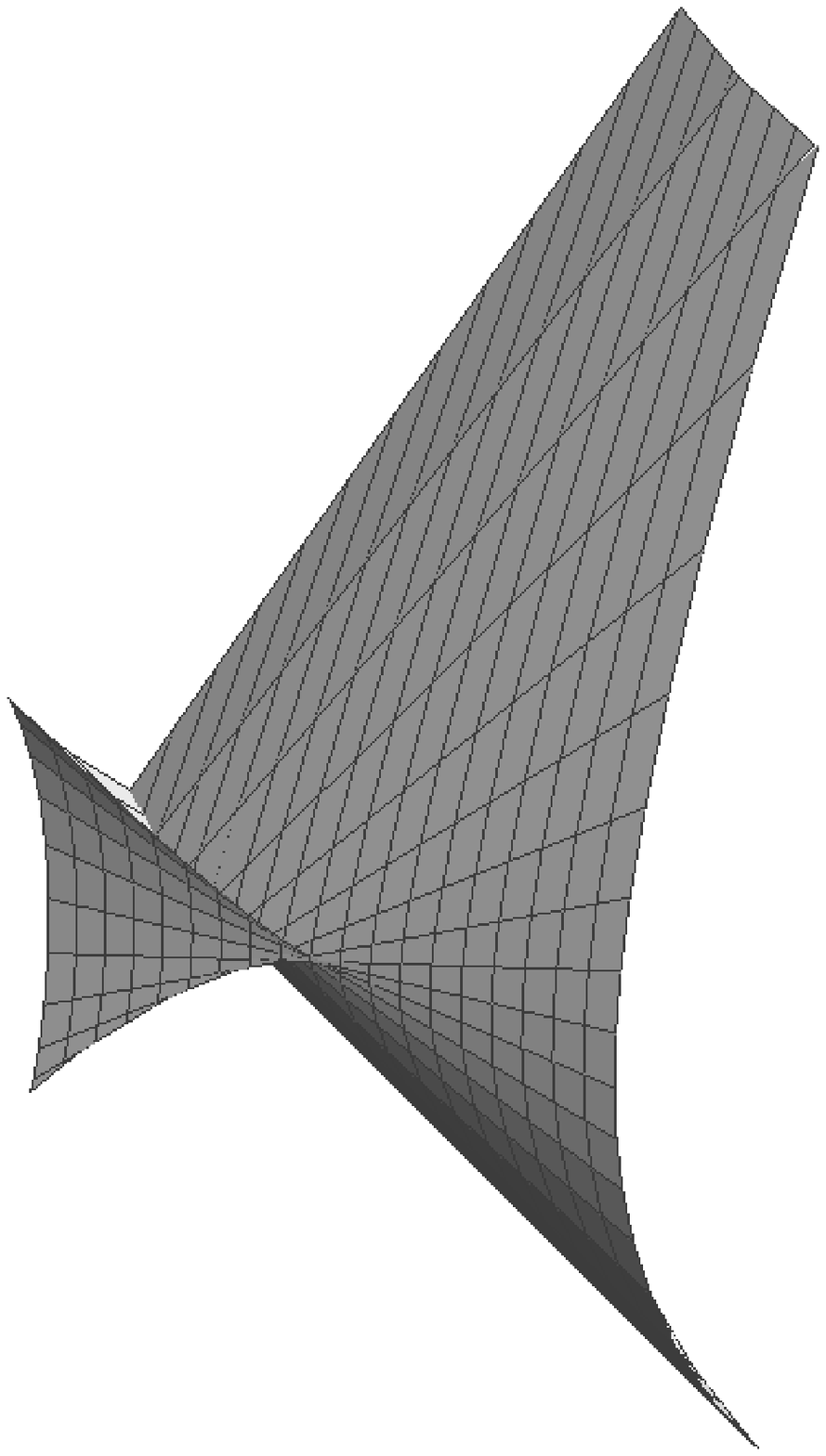}&
        \includegraphics[height=2.3cm]{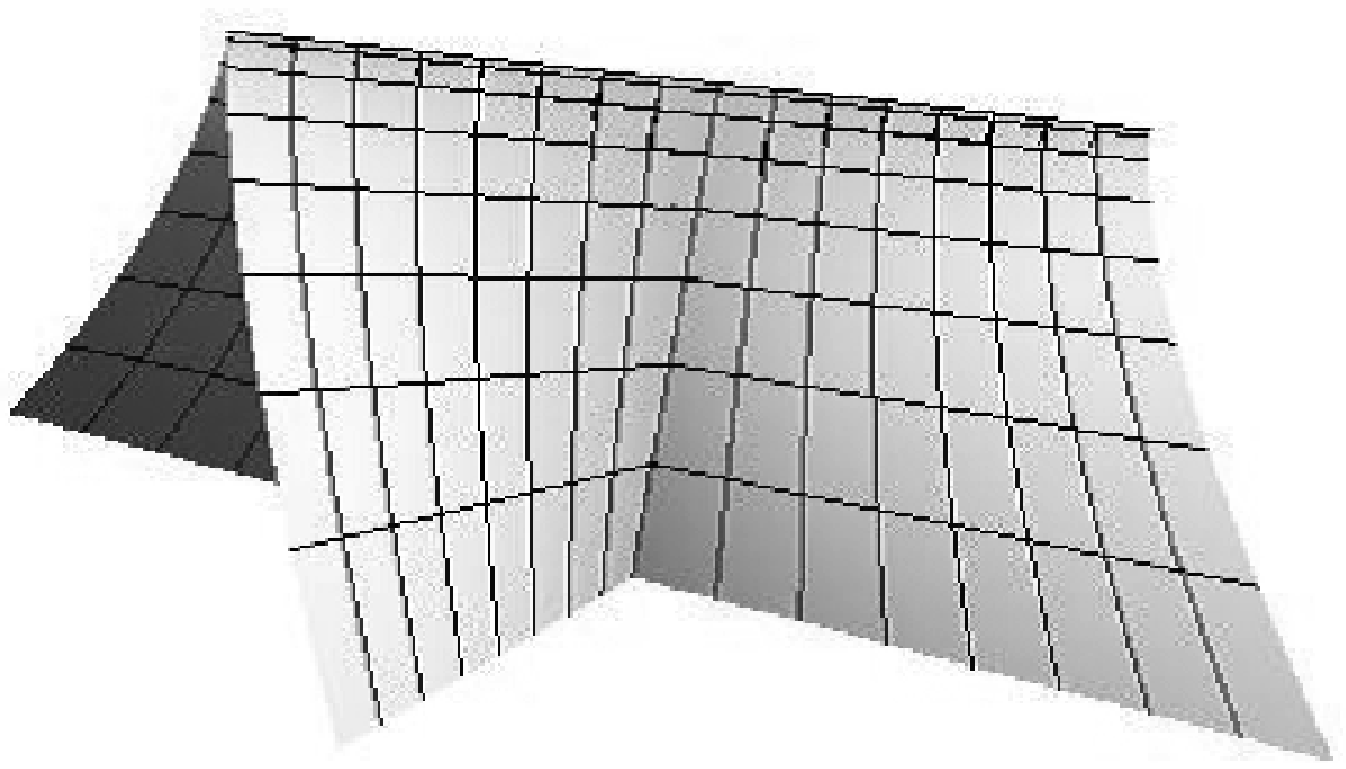}
  \end{tabular}
 \end{center}
  \caption{The cuspidal edge, the swallowtail and 
           the cross cap}
  \label{fig:ruled3}
\end{figure}
\red{
\begin{remark}[%
 Isometric deformations of ruled surfaces with singularities]
\label{rmk:ruled}
 Let $\gamma(t)$ be a curve in $\R^3$ defined near $t=0$,
 and $\xi(t)$ a vector field along the curve $\gamma$
 such that $\xi$ does not vanish and $\xi'(0)\ne 0$.
 Then the ruled surface $f(u,v):=\gamma(v)+u \xi(v)$
 has non-trivial isometric deformations as follows:
 By the coordinate change
 $(u,v)\mapsto (u/|\xi(v)|,v)$,
 we may assume that $|\xi(v)|=1$.
 Since $\xi'(0)\ne 0$, we may also  assume that $|\xi'(v)|=1$.
 Then $(\xi(v),\xi'(v),\xi(v)\times \xi'(v))$
 forms an orthonormal frame field, and the derivative of
 $\gamma(v)$ has the following expression
 \begin{equation*}
  \gamma'(v)=a(v)\xi(v)+b(v)\xi'(v)+c(v)\bigl(\xi(v)\times \xi'(v)\bigr).
 \end{equation*}
 Let $\tilde\xi(t)$ be an arbitrarily given spherical curve
 with arc length parameter, and 
 let $\tilde \gamma$ be the curve whose derivative is given as
 \[
    \tilde \gamma'(v)=a(v)\tilde\xi(v)+b(v)
     \tilde\xi'(v)+c(v)\bigl(\tilde\xi(v)\times \tilde\xi'(v)\bigr).
 \]
 Then 
 $\tilde f(u,v):=\tilde \gamma(v)+u \tilde \xi(v)$
 has the same first fundamental form as $f(u,v)$.
 By computing the second fundamental form,
 one can easily verify that $f$ and $\tilde f$
 are congruent if $\xi$ and $\tilde \xi$ are as well.
 This implies that several well-known singularities
 on surfaces admit isometric deformations like as
 in the case of cross caps
 in Theorem~\ref{thm:deform}. 
 For example,
\begin{it}
\renewcommand{\theenumi}{\rm{(\arabic{enumi})}}
\renewcommand{\labelenumi}{\rm{(\arabic{enumi})}}
\begin{enumerate}
 \item[{\rm (i)}] $(0,0)$ is a cross cap if and only if
       $\gamma'(0)=0$
       and $\det(\gamma''(0),\xi(0),\xi'(0))\ne 0$.
 \item[{\rm (ii)}]
  $f$ is a developable map {\rm (}i.e.\ the Gaussian curvature of $f$
       vanishes identically at each regular point of $f${\rm)}
       having a cuspidal edge 
       {\rm(}cf.\ Figure \ref{fig:ruled3}, left{\rm)} at
       $(0,0)$  if $b(v),\,\, c(v)$ vanish identically
       and $a(0)\ne 0$ hold
       {\rm(}cf.\ \cite[Fact (1)]{FSUY}{\rm)}.
 \item[{\rm (iii)}] $f$ is a developable map having a swallowtail 
       {\rm(}cf.\ Figure \ref{fig:ruled3}, center{\rm)}
       at $(0,0)$  if $b(v),c(v)$ vanish identically
       and $a(0)=0$, $a'(0)\ne 0$ hold
       {\rm(}cf.\ \cite[Fact (2)]{FSUY}{\rm)}.
 \item[{\rm (iv)}] We set $\nu=\xi\times \xi'$.
       Then $f$ is a developable map
       having a cuspidal cross cap 
       {\rm(}cf.\ Figure \ref{fig:ruled3}, right{\rm)}
       at $(0,0)$ if 
       $b(v)$ and $c(v)$ vanish identically, and
       \[
          \det(\xi(0),\nu(0),\nu'(0))=0,\quad
          a(0)\ne 0,\quad
          \det(\xi(0),\nu(0),\nu''(0))\ne  0
       \]
hold.
       This criterion can be proved 
       by applying \cite[Corollary 1.5]{FSUY},
       using the fact that
       $\nu$ gives the unit normal vector field
       when $f$ is a developable map.
\end{enumerate}%
\end{it}%
 One can choose $\gamma$ and $\xi$ so that
 they satisfy each of the above criteria. 
 This implies that cross caps, cuspidal edges, swallowtails,
 and cuspidal cross caps actually admit non-trivial
 isometric deformations
 in a certain class of ruled surfaces.
\end{remark}
}

Using the existence of non-trivial isometric deformations of 
degenerate quadratic cross caps, we can prove the following assertion.

\begin{theorem}\label{thm:extrinsic}
 The three invariants $a_{12}$, $a_{03}$ and $b_3$ are extrinsic.
\end{theorem}
\begin{proof}
 Let $c(s)$ be a spherical curve with arc length parameter $s$ so that
 \[
     c(0)=(1,0,0),\qquad 
     \frac{dc}{ds}(0)=\frac{1}{\sqrt{1+a^2_{11}}}(0,1,a_{11}),
 \]
 and let $\vect{e}:=dc/ds$, $\vect{n}:=c\times \vect{e}$.
 Then
 \[
    \frac{d\vect{e}}{ds}(s)=\kappa(s)\vect{n}(s)-c(s),\qquad
    \frac{d\vect{n}}{ds}(s)=-\kappa(s)\vect{e}(s),
 \]
 hold, where $\kappa(s)$ is the curvature function.
 Using these, one can see that the cross cap $f=f_c$ given in 
 Theorem~\ref{thm:deform}
 has the following expansion:
 \begin{multline}\label{eq:cubic}
    f(u,v)=
        \left(
          u, uv, a_{11}uv +\frac{1}{2}a_{02}v^2\right)\\
     +\frac{\kappa(0)\sqrt{1+a_{11}^2}}{6}
      \left(
         0 , -3a_{11}uv^2-2a_{02}v^3, 3uv^2
      \right)+O(u,v)^4.
 \end{multline}
 By a parameter change $v=w+\frac12a_{11}\kappa(0)\sqrt{1+a_{11}^2}\,w^2$,
 \eqref{eq:cubic} is rewritten as 
 \begin{multline*}
    f(u,w)=
        \left(
          u, uw, a_{11}uw +\frac{1}{2}a_{02}w^2\right)\\
     +\frac{\kappa(0)\sqrt{1+a_{11}^2}}{6}
      \left(
         0 , -2a_{02}w^3, 3 (1+a_{11}^2) uw^2 + 3 a_{11}a_{02}w^3
      \right)+O(u,w)^4.
 \end{multline*}
 Thus, $(u,w)$ forms the canonical coordinate system up to 
the third order terms,
and the invariants $a_{12}$, $a_{03}$ and $b_3$ are expressed as 
\begin{equation}\label{eq:extinv}
\begin{gathered}
    a_{12}=\kappa(0)(1+a_{11}^2)^{3/2},\qquad
    a_{03}=3a_{02}a_{11}\kappa(0)\sqrt{1+a_{11}^2},\\
    b_3 = -2a_{02}\kappa(0)\sqrt{1+a_{11}^2}
\end{gathered}
\end{equation}
 which depend on the initial value $\kappa(0)$ of the
 geodesic curvature function, 
 and thus they are extrinsic.
\end{proof}
\begin{remark}
 By a straightforward calculation,
 one can also check  that
 $a_{0j}$, $a_{1j}$ and $b_j$
 for $j=3,4,5,\dots$
 all changes values by the same deformation
 as in the proof of Theorem \ref{thm:extrinsic}.
\end{remark}

\section{Differential geometry of cross caps}
\label{sec:geometry}

 Let $f:U\to \R^3$ be a $C^\infty$-map and $p\in U$ 
 a cross cap singularity.
 A local coordinate system $(u,v)$ centered at $p$ is said to be 
 {\it admissible\/} if it satisfies $f_v(0,0)=0$.
Canonical coordinate systems of cross caps are admissible.
The concept of admissible coordinate systems is intrinsic, 
since $\partial/\partial v$ at $(0,0)$ points the degenerate 
directions of the induced metrics. 
In contrast to Theorem~\ref{thm:extrinsic}, the following assertion 
holds:

\begin{theorem}\label{thm:intrinsic}
 The coefficients $a_{02}$, $a_{20}$ and $a_{11}$
 are intrinsic invariants of cross caps.
\end{theorem}
\begin{proof}
 Let $(0,0)$ be a cross cap singularity of a $C^\infty$-map 
 $f:(U;u,v) \to \R^3$,
 and $(u,v)$ be an admissible coordinate system.
 Without loss of generality, we may assume that
 \begin{equation*}
     [f_u,f_{uv},f_{vv}]>0
 \end{equation*}
 by applying the coordinate change $(u,v)\mapsto (-u,-v)$
 if necessary, where 
 \[
    [\vect{a},\vect{b},\vect{c}]:=
        \det(\vect{a},\vect{b},\vect{c})
        =(\vect{a}\times\vect{b})\cdot \vect{c}
 \]
 and $\times$ is the vector product of $\R^3$.

 Then we have that
 \begin{align}
  \label{eq:a02b}
    a_{02}&=\frac{|f_u|\,|f_u\times f_{vv}|^3}{[f_u,f_{uv},f_{vv}]^2}, \\
  \label{eq:a20}
    a_{20}&=\frac{|f_u\times f_{vv}|}{4|f_u|^3[f_u,f_{uv},f_{vv}]^2} 
        \biggl( 
            [f_u,f_{uu},f_{vv}]^2\\
        &\hspace{0.4\textwidth}
               +4 [f_u,f_{uv},f_{vv}][f_u,f_{uv},f_{uu}]
        \biggr), \nonumber\\
  \label{eq:a11}
    a_{11}&=\frac{1}{2|f_u|[f_u,f_{uv},f_{vv}]^2} 
        \biggl( 
          2[f_u,f_{uv},f_{vv}]\det\pmt{f_u\cdot f_u & f_u\cdot f_{uv} 
           \\ f_{vv}\cdot f_u & f_{vv}\cdot f_{uv}} \\
  &\hspace{0.45\textwidth}
           -|f_u\times f_{vv}|^2[f_u,f_{uu},f_{vv}] \nonumber 
         \biggr)
 \end{align}
 hold at $(u,v)=(0,0)$.
 One can prove these identities immediately: 
 In fact, the right-hand sides of these identities are 
 independent of the choice of admissible coordinate systems,
 and these identities themselves can be directly verified 
 for the canonical coordinate system of $f$.

 We now set
 \[
     E:=f_u\cdot f_u,\quad
     F:=f_u\cdot f_v,\quad
     G:=f_v\cdot f_v,
 \]
 which are the coefficients of the induced metric of the cross cap. 
 It is sufficient to show that the right-hand sides of \eqref{eq:a02b},
 \eqref{eq:a20} and \eqref{eq:a11} are
 written in terms of derivatives of $E$, $F$ and $G$
 at $(0,0)$.

 We first show that $a_{02}$ is intrinsic:
 Since $f_v(0,0)=0$, it holds that
 \begin{equation}\label{eq:delta}
    [f_u,f_{uv},f_{vv}]^2
     =\det\left(\pmt{f_u\\ f_{uv}\\ f_{vv}}(f_u,f_{uv},f_{vv})\right)
     =
     \det\pmt{E    & F_u      & F_v \\
              F_u  & G_{uu}/2 & G_{uv}/2 \\
              F_v  & G_{uv}/2 & G_{vv}/2}
 \end{equation}
 at $(u,v)=(0,0)$,
 where we used the identities
 \begin{align*}
    &f_u\cdot f_u=E,\quad 
     f_u\cdot f_{uv}=F_u,\quad 
     f_u\cdot f_{vv}=F_v,\\
    &f_{uv}\cdot f_{uv}=\frac{G_{uu}}2,\quad 
     f_{uv}\cdot f_{vv}=\frac{G_{uv}}2,\quad 
     f_{vv}\cdot f_{vv}=\frac{G_{vv}}2
 \end{align*}
 at $(u,v)=(0,0)$.
 Since
 \[
     |f_u|^2=E,\quad
     |f_u\times f_{vv}|^2=
            (f_u\cdot f_u)(f_{vv}\cdot f_{vv})-(f_u\cdot f_{vv})^2
           =\frac{EG_{vv}}2-(F_{v})^2
 \]
 at $(u,v)=(0,0)$, we can conclude that $a_{02}$ is an intrinsic invariant.

 Similarly, to prove $a_{20}$ and $a_{11}$ are intrinsic,
 it is sufficient to show that
 $[f_u,f_{uu},f_{vv}]$ and $[f_u,f_{uv},f_{uu}]$
 are both written in terms of derivatives of $E$, $F$ and $G$ at $(0,0)$.
 In fact, \eqref{eq:delta} implies that
 $[f_u,f_{uu},f_{vv}]$ is intrinsic, because 
 \begin{align*}
    [f_u,f_{uu},f_{vv}]
      &=\frac1{[f_u,f_{uv},f_{vv}]}
         \det\left(\pmt{f_u\\ f_{uv}\\ f_{vv}}
            (f_u,f_{uu},f_{vv})\right)\\
      &=\frac1{[f_u,f_{uv},f_{vv}]}
 \det\pmt{f_u\cdot f_u & f_u\cdot f_{uu}& f_u\cdot f_{vv}\\
    f_{uv}\cdot f_u& f_{uv}\cdot f_{uu}& f_{uv}\cdot f_{vv} \\
    f_{vv}\cdot f_u & f_{vv}\cdot f_{uu}& f_{vv}\cdot f_{vv}}\\
      &=\frac1{[f_u,f_{uv},f_{vv}]}
         \det\pmt{E &   f_u\cdot f_{uu} & F_v\\
              F_u  & f_{uv}\cdot f_{uu} & G_{uv}/2 \\
              F_v  & f_{vv}\cdot f_{uu}& G_{vv}/2}
 \end{align*}
 holds at $(u,v)=(0,0)$, and
 \begin{align*}
     &f_u(0,0)\cdot f_{uu}(0,0)=\frac{E_u(0,0)}2,\\
     &f_{uv}(0,0)\cdot f_{uu}(0,0)=F_{uu}(0,0)-\frac{E_{uv}(0,0)}2,\\
     &f_{vv}(0,0)\cdot f_{uu}(0,0)=F_{uv}(0,0)-\frac{E_{vv}(0,0)}2.
 \end{align*}
 Similarly, $[f_u,f_{uv},f_{uu}]$ is intrinsic, 
 because of the identity
 \[
    [f_u,f_{uv},f_{uu}]
         =\frac1{[f_u,f_{uv},f_{vv}]}
           \det\left(\pmt{f_u\\ f_{uv}\\ f_{vv}}
          (f_u,f_{uv},f_{uu})\right).
 \]
\end{proof}

\begin{remark}
The value $\Delta:=[f_u(0,0),f_{uv}(0,0),f_{vv}(0,0)]$
is a criterion of cross cap singularities.
In the above proof (cf. \eqref{eq:delta}), we showed the identity
$$
\Delta^2=
     \det\pmt{E    & F_u      & F_v \\
              F_u  & G_{uu}/2 & G_{uv}/2 \\
              F_v  & G_{uv}/2 & G_{vv}/2}
$$
at $(u,v)=(0,0)$, which implies that $\Delta$
is intrinsic. Moreover, we set
$$
h(u,v):=E(u,v)G(u,v)-(F(u,v))^2.
$$
Using the fact that $(u,v)$ is admissible,
one can easily prove
$$
\Delta^2=\frac1{4E(0,0)}\biggl(h_{uu}(0,0)h_{vv}(0,0)-(h_{uv}(0,0))^2
\biggr),
$$
that is, $\Delta$ is closely related to the Hessian of $h$.
Since $h(u,v)$ is non-negative and $h(0,0)=0$,
it holds that $h_{vv}(0,0)\ge 0$.
Moreover, the identity 
$$
a_{02}=\frac{\sqrt{E(0,0)}(h_{vv}(0,0))^{3/2}}{2\Delta^2}
$$
holds.
\end{remark}

In \cite{W}, \cite{T} and \cite{FH},
ellipticity, hyperbolicity and parabolicity of cross caps are defined.
The following assertion holds: 
\begin{corollary}
 The ellipticity, hyperbolicity and parabolicity of 
 cross caps in $\R^3$ are all intrinsic properties.
\end{corollary}
\begin{proof}
 A cross cap is elliptic (resp.\ hyperbolic) if $a_{20}>0$
 (resp.\ $a_{20}<0$).
 Since we have already seen that $a_{20}$ is intrinsic, 
 ellipticity and hyperbolicity are as well.
 In \cite{FH}, it was shown that a cross cap is parabolic 
 if and only if $a_{20}=0$ and the zero set $Z_K$
 of the Gaussian curvature  gives a regular curve in 
 the $r\theta$-plane which is tangent to the line $r=0$,
 where 
 \[
    u=r\cos \theta,\qquad v=r \sin \theta 
 \]
 and $(u,v)$ is a canonical coordinate system.
 Since the set $Z_K$ is intrinsic and this tangency
 property does not depend  on the choice of an admissible 
 coordinate system, we get the assertion.
\end{proof}

We fix a cross cap $f:(U;u,v)\to \R^3$,
where $(u,v)$ is an admissible coordinate system, 
that is, $(u,v)=(0,0)$ is a cross cap singularity and $f_v(0,0)=0$.
We call the line
\[
   \{f(0,0)+t f_u(0,0)\,;\, t\in \R\}
\]
the {\it tangential line\/} at the cross cap and a non-zero vector at 
$T_{f(0,0)}\R^3$ proportional to $f_u(0,0)$ is called the
{\it tangential direction}.
The plane passing through $f(0,0)$ spanned by $f_u(0,0)$ and 
$f_{vv}(0,0)$ is called the {\it principal plane}.
On the other hand,
the plane passing through $f(0,0)$ perpendicular to 
the proper tangential direction is called the {\it normal plane}.
The unit normal vector $\nu(u,v)$ near the cross cap at $(u,v)=(0,0)$ 
can be extended as a $C^\infty$-function of $r$, $\theta$ 
by setting $u=r \cos \theta$ and $v=r \sin \theta$, and the
limiting normal vector
\begin{equation*}
    \nu(\theta):=\lim_{r\to 0}\nu(r \cos \theta,r \sin \theta)
                         \in T_{f(0,0)}\R^3
\end{equation*}
lies in the normal plane.
Then, one can consider the parallel family of cross caps
\begin{equation}\label{eq:parallel}
    f_t(r,\theta):=f(r,\theta)+t \nu(r,\theta)
\end{equation}
which are $C^\infty$-maps with respect to $(r,\theta)$,
even at $r=0$.
The focal surface of this parallel family meets the normal plane 
at the focal conic as mentioned in the introduction.
On the other hand, the principal plane has the following property:
\begin{proposition}
 The initial velocity vector of the space curve emanating 
 from the cross cap singularity which parametrizes the 
 self-intersection is contained in the principal plane.
\end{proposition}

\begin{proof}
 Since the principal plane is invariant under diffeomorphisms of $\R^3$, 
 this assertion can be verified by the standard cross cap.
\end{proof}

\red{To get much precise information to the set of
cross caps, we give the following definition:}

\red{
\begin{definition}
A germ of cross cap $f:U\to \R^3$ is called {\it normal}
if the set of self-intersections is contained 
in the intersection of the principal plane and
the normal plane.
\end{definition} 
}

\red{The quadratic cross caps defined in Section 2
are all normal. 
We get the following criterion of normal cross caps:}

\red{
\begin{theorem}\label{thm:normal}
The germ of a real analytic 
cross cap is normal if and only if
all of the invariants $(b_j)_{j=3,4,5,\cdots}$
associated to its normal form \eqref{eq:cross} 
vanishes simultaneously.
\end{theorem}
}
\red{
\begin{proof}
Without loss of generality, we may assume that 
the given cross cap $f$ has an expression as in
\eqref{eq:cross}. 
We set
$$
\beta(v):=\sum_{i=3}^\infty \frac{b_iv^i}{i!}.
$$
Since $f$ is real analytic, $\beta$ is a real analytic
function. If $\beta$ vanishes identically, then
$$
f(0,v)
=\left(0,0,
\sum_{n=2}^\infty 
\frac{a_{0n}v^{n}}{n!}\right).
$$
Since $a_{02}>0$, 
$$
w:=\sqrt{\sum_{n=2}^\infty 
\frac{a_{0n}v^{n}}{n!}}
$$
is well-defined and gives a real analytic
function. 
Replacing the coordinate system $(u,v)$ by
$(u,w)$,
we have that
$$
f(0,w)
=(0,0,w^2)=f(0,-w),
$$
which implies that 
the set of self-intersection of $f$
lies in the third axis,
namely,
the set of self-intersections is contained 
in the intersection of the principal plane and
the normal plane.
Conversely, we assume that the set $S$ of self-intersection
lies in the third axis.
Then the first and the second components
of \eqref{eq:cross} yield that
$$
u=0,\qquad uv+\beta(v)=0
$$
hold along $S$. 
Thus the $v$-axis parametrizes the set $S$ and
$\beta(v)$ vanishes identically,
which proves the assertion.
\end{proof}
}

As pointed out in the introduction, $a_{20}$ is 
an important intrinsic invariant of cross caps related
to the sign of the Gaussian curvature.
The following assertion can be proved easily,
which gives a geometric meaning for $a_{20}$:
\begin{proposition}
 The section of a cross cap by its principal plane
contains a regular curve $\gamma$ whose velocity vector
 is $f_u(0,0)$, and then the curvature of
 $\gamma$ as a plane curve at the cross cap is equal to $a_{20}$, 
 where we give the orientation to the principal
 plane so that $\{f_u(0,0), f_{vv}(0,0)\}$ is a positive frame.
\end{proposition}

Intersections of a cross cap with planes are discussed in
\cite{FH2}.  

\medskip
At the end of this section, we discuss on cross caps
in an arbitrary Riemannian $3$-manifold $(N^3,g)$.
Let $f:M^2\to (N^3,g)$ be a $C^\infty$-map having 
a cross cap singularity at $p\in M^2$, 
where $M^2$ is a $2$-manifold.
Then there exists a local coordinate system $(u,v)$ 
of $M^2$ centered at $p$ 
and a normal coordinate system $(x,y,z)$ of $(N^3,g)$
centered at $f(p)$ such that (cf.\ \eqref{eq:cross})
\[
   f(u,v)=
   \left(u,uv,\frac{a_{20}}{2}u^2+a_{11}uv+\frac{a_{02}}{2}v^2
   \right)+O(u,v)^{3}.
\]
Like as in the case of the Euclidean $3$-space,
one can easily verify that
$a_{20}$, $a_{02}$ and $a_{11}$ are all intrinsic invariants:
In fact, we set
\[
   [\vect{a},\vect{b} ,\vect{c}]:=\Omega(\vect{a},\vect{b} ,\vect{c})
        =g(\vect{a}\times \vect{b},\vect{c}),
\]
where $\Omega$ is the Riemannian volume form of $(N^3,g)$
and $\vect{a},\vect{b},\vect{c}$ are vector fields
of $N^3$ along the $C^\infty$-map $f$.
We denote by $D$ the Levi-Civita connection of $g$. 
By replacing
\[
    f_{uu}\mapsto D_{u}f_u,\quad
    f_{uv}\mapsto D_{u}f_v,\quad
    f_{vv}\mapsto D_{v}f_v,
\]
the three formulas \eqref{eq:a02b}, \eqref{eq:a20}
and \eqref{eq:a11} hold at the cross cap
singularity of $N^3$.

By a straightforward calculation,
the first and the second fundamental forms
\[
  E\, du^2+2F\,du\,dv+G\,dv^2, \quad
  L\, du^2+2M\,du\,dv+N\,dv^2 
\]
have the following expressions
\begin{align*}
   &E=1+r^2 \bigl(
               \sin^2\theta+(a_{20}\cos \theta+a_{11}\sin \theta)^2
               +O(r)\bigr), \\
   &F=r^2 \left(a_{20}a_{11}\cos^2\theta+
          (1+a_{20}a_{02}+a_{11}^2)\sin\theta\cos \theta
      +a_{11}a_{02}\sin^2 \theta+O(r)\right), \\
   &G=r^2\bigl(A^2_\theta+O(r)\bigr), \\
   &L=\frac{a_{20}\cos \theta}{A_\theta}+O(r),\quad
    M=-\frac{a_{02}\sin \theta}{A_\theta}+O(r),\quad
    N=\frac{a_{02}\cos \theta}{A_\theta}+O(r),
\end{align*}
where $u=r \cos\theta$, $v=r\sin \theta$ and
\[
   A_\theta:=\sqrt{\cos^2\theta+(a_{11}\cos\theta+a_{02}\sin\theta)^2}.
\]
We denote by $K_{\ext}$ the determinant of the
shape operator of $f$, which is called the 
{\em extrinsic curvature function}.

Since $EG-F^2=r^2(A^2_\theta+O(r))$, the mean curvature
function $H$,
the extrinsic curvature function $K_{\ext}$,
and the Gaussian curvature function $K$
are given by
\begin{align}
   H&=\frac{1}{r^2}
      \left(\frac{a_{02}\cos \theta}{2A_{\theta}^{3}}+
             O(r)
      \right),
      \label{eq:H-N}
    \\
   K_{\ext}&
     =\frac{a_{02}}{r^2A_\theta^4}\left(a_{20}\cos^2\theta-a_{02}\sin^2
            \theta+O(r)\right),
    \label{eq:K-N}\\
   K&=K_{\ext}+c_{g}(r,\theta)=
      \frac{a_{02}}{r^2A_\theta^4}\left(a_{20}\cos^2\theta-a_{02}\sin^2
     \theta+O(r)\right),
\label{eq:Kext-N}
\end{align}
where $c_{g}(r,\theta)$ is 
a suitable $C^\infty$-function at $p$
with respect to the sectional
curvature of the Riemannian metric $g$
appeared in the Gauss equation.
When $(N^3,g)$ is the Euclidean space, 
the formulas \eqref{eq:H-N} and \eqref{eq:K-N}
(and also the description of principal curvatures) have been
given in \cite{FH}. 
It should be remarked that
the top terms of the curvature functions $K$ and $H$ are determined
by the three invariants $a_{20}$, $a_{02}$ and $a_{11}$.
This fact seems a remarkable property of cross caps since 
{\it the top terms of $H$ and $K$ do not depend on a choice 
     of ambient spaces}.
Moreover, \eqref{eq:K-N} and \eqref{eq:Kext-N} imply that
the asymptotic behaviors of $K$ and $K_{\ext}$ are same at the cross cap
singularity.

The following is a generalization of the assertion
proved in Fukui-Ballesteros \cite{FB}
and  Tari \cite{T} when $(N^3,g)$ is the  Euclidean $3$-space:
\begin{proposition}
 Umbilical points do not accumulate to a cross cap
 in $(N^3,g)$.
\end{proposition}
\begin{proof}
 By \eqref{eq:K-N},
 we know that $K_{\ext}<0$ if $\theta=\pm \pi/2$.
 Thus it is sufficient to show that $H^2-K_{\ext}$ does not vanish
 under the assumption $\cos \theta\ne 0$.
 In fact
 \[
    H^2-K_{\ext}=
    \frac1{r^4}\left(\frac{(a_{02}\cos \theta)^2}{4A_{\theta}^{6}}+O(r)\right)
 \]
diverges if $\cos \theta\ne 0$ as $r$ tends to zero.
\end{proof}

As noted in the introduction, the ellipticity and hyperbolicity of
cross caps are determined by the sign of the invariant $a_{20}$.
In this paper, we have shown the 
existence of non-trivial isometric deformations of
quadratic cross caps when $a_{20}$ vanishes.

\red{
\section{Invariants of cross caps under isometric deformations}
\label{sec:deform2}
It was classically known that 
regular surfaces (not only ruled surfaces) 
admit non-trivial isometric deformations in general,
and such deformations can be expected even at cross cap singularities.
In this section, we shall give further invariants
under isometric deformation of cross caps.
The following assertion holds:
\begin{proposition}
 The four \red{quantities} written in terms of coefficients
 of the normal form as in \eqref{eq:cross}
 \begin{align}\label{eq:a}
  a_{03} + \frac{3a_{11} b_3}2,\quad
  a_{12} + \frac{(1 + a_{11}^2) b_3}{2a_{02}} ,\\
\label{eq:b}
  a_{21} - \frac{a_{11} a_{20} b_3}{6a_{02}},\quad
  a_{30} - \frac{(1 + a_{11}^2) a_{20} b_3}{2a_{02}^2} 
 \end{align}
 are common for two cross caps having the same 
 first fundamental form.
 In particular, they do not change
 under isometric deformations of cross caps.
\end{proposition}
\begin{proof}
 Let $f_0$ and $f_1$ be two cross caps having the following normal
 forms respectively;
 \begin{align}
    f_0(u,v)&=
        \left(
	   u,
	   uv+\frac{b_3}{3!}v^3,
	   \sum_{r=2}^3 \sum_{j=0}^r \frac{a_{j\,r-j}}{j!(r-j)!}u^j
	   v^{r-j}
	\right) +O(u,v)^{4}, \\
    f_1(x,y)&=
        \left(
	   x,
	   xy+\frac{B_3}{3!}y^3,
	   \sum_{r=2}^3 \sum_{j=0}^r \frac{A_{j\,r-j}}{j!(r-j)!}x^j
	   y^{r-j}
	\right) +\red{O(x,y)^{4}}.
 \end{align}
 Since $(x,y)$ and $(u,v)$ are both local coordinate systems of
 $\R^2$, the mapping $(u,v)\mapsto (x(u,v),y(u,v))$
 is a diffeomorphism.
 Since two coordinates give a normal form of
 $f_0$ and $f_1$, we may set
 \[
       x_u(0,0)=y_v(0,0)=1,\qquad x_v(0,0)=y_u(0,0)=0.
 \]
 Suppose that $f_0$ and $f_1$ share the same first fundamental form.
 Since we have seen that $a_{02}$, $a_{20}$ and $a_{11}$ are
 intrinsic (by Theorem~\ref{thm:intrinsic}), we have that
 \[
    a_{20}=A_{20},\qquad a_{11}=A_{11},\qquad a_{02}=A_{02}.
 \]
 Moreover, it holds that
 \begin{align}\label{eq:e1}
  E_0&=E_1(x_u)^2+2F_1 x_uy_u+G_1 (y_u)^2, \\
  \label{eq:f1}
  F_0&=E_1 x_ux_v+F_1 (x_uy_v+x_vy_u)+G_1 y_uy_v, \\
  \label{eq:g1}
  G_0&=E_1(x_v)^2+2F_1 x_vy_v+G_1 (y_v)^2, 
 \end{align}
 where
 \begin{equation}
   E_i:=(f_i)_u\cdot (f_i)_u,\quad
   F_i:=(f_i)_u\cdot (f_i)_v,\quad
   G_i:=(f_i)_v\cdot (f_i)_v \qquad (i=0,1).
 \end{equation}
 Computing the first and second order terms
 of the Taylor expansions of the left and right-hand sides 
 of \eqref{eq:e1}, \eqref{eq:f1} and \eqref{eq:g1},
 we get the following relations:
 \begin{gather*}
     x_{uu}(0,0)=x_{uv}(0,0)=x_{vv}(0,0)=0,\\
     x_{uuu}(0,0)=x_{uuv}(0,0)=x_{uvv}(0,0)=x_{vvv}(0,0)=0.
 \end{gather*}
 Similarly, computing the Taylor expansions of
 the third order derivatives
 \[
    \partial^3/\partial u^3,\quad
    \partial^3/\partial u^2\partial v,\quad
    \partial^3/\partial v \partial u^2,\quad
    \partial^3/\partial v^3,
 \]
one gets explicit expressions for 
$y_{uu}(0,0)$, $y_{uv}(0,0)$, $y_{vv}(0,0)$, and
 \[
  a_{03}-A_{03},\quad a_{12}-A_{12},\quad a_{21}-A_{21},
   \quad a_{30}-A_{30}
 \]
 in terms of $a_{02}$, $a_{11}$, $a_{20}$ and $b_3-B_3$,
 which proves the assertion.
\end{proof}
\begin{remark}
 The above conclusion implies that $a_{30}$ and $a_{21}$ do not change 
 under isometric deformations when $a_{20}=0$.
 Moreover, \eqref{eq:a} and \eqref{eq:b} imply that
 \begin{equation}\label{eq:relation}
   2 a_{03}+3a_{11}b_3=0,\qquad
   2 a_{02}a_{12}+(1+a_{11}^2)b_3=0
 \end{equation}
 hold for non-trivial isometric deformations of quadratic
 cross caps even when $a_{20}\ne 0$.
 The relations \eqref{eq:relation}
 also follow from \eqref{eq:extinv} in the case $a_{20}=0$.
\end{remark}
One can continue the same calculation for the fourth order terms. 
The authors checked using Mathematica that
\[
  a_{04}-A_{04},\quad
  a_{13}-A_{13},\quad
  a_{22}-A_{22},\quad
  a_{31}-A_{31},\quad
  a_{40}-A_{40}
\]
can be written in terms of
$a_{02}$, $a_{11}$, $a_{20}$, \red{$b_3$}, $b_3-B_3$ and $b_4-B_4$.
Using this, it can be observed that
$a_{ij}=A_{ij}$ ($2\le i+j\le 4$)
hold if $b_3=B_3$ and $b_4=B_4$.
}

\section*{Acknowledgments}
The authors thank Shyuichi Izumiya, Toshizumi Fukui and Wayne Rossman 
for valuable comments. \red{The fourth author thanks 
Huili Liu for fruitful discussions on this subject at 8th 
Geometry Conference for Friendship of China and Japan
at Chengdu.}

\end{document}